\documentclass[twoside]{amsart}
\usepackage{amssymb}
\usepackage{amsmath}
\usepackage{amsthm}
\usepackage[all]{xy}

\newtheorem{theorem}{Theorem}[section]
\newtheorem{proposition}[theorem]{Proposition}

\newtheorem{principle}[theorem]{Principle}
\newtheorem{lemma}[theorem]{Lemma}
\theoremstyle{definition}
\newtheorem{definition}[theorem]{Definition}

\theoremstyle{remark}
\newtheorem{remark}[theorem]{Remark}

\newenvironment{eq}{\addtocounter{theorem}{1}\begin{equation}}{\end{equation}}

\newcommand{\cC}{{\mathcal C}}

\newcommand{\Z}{\mathbb Z}
\newcommand{\C}{\mathbb C}

\newcommand{\bP}{\mathbb P}

\newcommand{\ra}{\rightarrow}
\newcommand{\lra}{\longrightarrow}
\newcommand{\iso}{\simeq}
\newcommand{\lin}{\sim}

\def\a{{\mathfrak a}}

\begin{document}

\title{On the Kummer construction}

\author{Marco Andreatta and Jaros{\l}aw A. Wi\'sniewski}

\thanks{Research of the first author was supported by grants of
  Italian Mur-PRIN. Research of the second author was financed by a
  grant of Polish MNiSzW (N N201 2653 33). Moreover, the second author
  thanks University of Trento for supporting his visits. Both authors
  thank the MSRI in Berkeley for hospitality during the last revision
  of the paper. Many thanks to Andrzej Weber for discussions. }

\address{Dipartimento di Matematica, Universita di Trento, I-38050
Povo (TN)} \email{marco.andreatta@unitn.it}

\address{Instytut Matematyki UW, Banacha 2, PL-02097
Warszawa} \email{J.Wisniewski@mimuw.edu.pl}



\begin{abstract}
We discuss a generalization of the Kummer construction. Namely 
an integral representation of a finite group produces an action on an abelian
variety and, via a crepant resolution
of the quotient, this gives rise to a higher dimensional variety with trivial
canonical class and first cohomology. We use virtual Poincar\'e polynomials with
coefficients in a ring of representations and McKay correspondence to
compute cohomology of such Kummer varieties.
\end{abstract}

\maketitle
\vskip 1cm
\section{Introduction}

Kummer surfaces are constructed in a two step process: (1) divide an
abelian surface by an action of an
involution, (2) resolve singularities of the quotient, which arise
from the fixed points of the action, by blowing them up to
$(-2)$-curves. The result of this process is a K3 surface, this is
because the group action kills the fundamental group of the abelian
surface and preserves the canonical form, and also because the resolution is
crepant. The invariants of this surface can be computed by looking at
the invariants of the involution and the contribution of the
resolution. This construction is classical, see \cite{Kummer} or
\cite{BarthPetersVandeVen}.

It is natural to ask about a generalization of the above procedure.
This involves dividing an
abelian variety by an action of a finite group. Our set up is as
follows:
\begin{itemize}
\item $G$ is a finite group with an irreducible integral representation 
$\rho_\Z: G\ra GL(r,\Z)$ whose fixed point set is
$\{0\}$,
\item $A$ is a complex abelian variety of dimension $d$, with neutral
element, addition and substraction denoted by $0$, $\pm$; note that the construction can 
be carried over starting with a compact complex torus as well.
If $d$ is odd we assume additionally $det(\rho_\Z)=1$, that is
$\rho_\Z: G\ra SL(r,\Z)$.
\end{itemize}

The first step of the generalized Kummer construction is achieved by
the induced action 
\begin{eq}\rho_A=\rho_\Z\otimes_\Z A:\ \ G\lra Aut(A^r)\end{eq}
which is obtained by identification $A^r=\Z^r\otimes_\Z A$. In other
words, $G$ acts on $A^r$ with integral matrices coming from the
representation $\rho_\Z$ and we consider the quotient $Y:= A^r / G$.

If a crepant resolution of $Y$ exists, $X \ra Y$,  we obtain a manifold $X$ of dimension $rd$
with $K_X\lin 0$ and $H^1(X,\C)=0$, as in the original Kummer
construction. The former comes from our assumptions regarding both the
action and the resolution. The latter follows because
$H^1(A^r/G,\C)=H^1(A^r,\C)^G=(\C^{2rd})^G$ is trivial by our
assumptions and the crepant resolution of $A^r/G$ does not change first
cohomology, as one can see using for instance the Leray spectral sequence and the 
rationality of quotient singularities.

\smallskip
Computing higher cohomology of $X$ in general is a hard task and we need the following extra properties of 
the (crepant) resolution. 
Let $\a$ be the tangent space of $A$ at identity (or at any point $p$), i.e. the Lie algebra of
holomorphic vector fields tangent to $A$. The induced action $\rho_\a=\rho_\Z\otimes_\Z\a:
G\ra GL(\a^r)$, which splits into $d$ copies of complexified representation
$\rho_\C=\rho_\Z\otimes_\Z\C$, is the tangent action. 
For any point $p \in A^r$ with non trivial isotropy group $G_p$ the action
${\rho_\a}_{| G_p} = {d\cdot\rho_\C}_{|G_p}$ is a representation of  $T_pA^r \iso \C^{dr}$.
In analytic or \'etale topology the action of
$G_p$ around $p$ is equivalent to the action
${\rho_\a}_{| G_p}$ in a neighborhood of $0$ in  $T_pA^r \iso \C^{dr}$.
For more details on the notation see the next section.

We will assume that:

a)   Over the set of points of
$A^r/G$ which represent orbits with the same isotropy the resolution
is a locally product (see the definition \ref{locally-product-resolution}).

b) McKay correspondence holds for the crepant resolutions of a
quotient singularity $\C^{rd}/G_p$; that is there exists a
canonical relation between the conjugacy classes of elements in $H= G_p$
and the cohomology  of the crepant resolution (see e.g.~\cite{ReidBourbaki},
\cite{BertinMarkushevich}, \cite{GiKa}).

\smallskip
By standard group action arguments we will have a description of the action of
$G$ on the set of points with the same isotropy (see \ref{structure-of-strata});
then we will glue these local resolutions to a global projective
resolution $X\ra Y= A^r/G$. This step may be non-obvious in case when the
singular points are non-isolated and their resolution is not obtained
in a canonical way; here we impose the assumption that  the resolution
is a locally product. 
Under the above assumptions, at the end of this procedure, we obtain a general formula to compute the
Poincar\'e polynomial $P_X(t)$ of $X$, that is the Betti numbers of $X$, as summarized in \ref{principle}. 

\medskip
In the second part of the paper we consider some examples which satisfy our assumptions.
In particular in section \ref{elliptic} we start with $d=1$, i.e.~$A$ is an elliptic curve, and we apply the construction to some representations in $SL(2, \Z)$ and in $SL(3, \Z)$. We compute in these cases,  with our procedure, the cohomolgy of the resulting Kummer surfaces and Calabi Yau threefolds.  A paper by Maria Donten \cite{Donten} provides a complete classification of Kummer 3-folds.

In section \ref{Got} we take $d=2$, i.e an abelian surface, and we consider the standard representation of the symmetric group $S_n$ in $SL(n-1, \Z)$; in this case the Kummer construction will produce the so called generalized Kummer manifolds, $Kum^{(n-1)}$, introduced by Beauville 
and Fujiki (see \cite{BeauvilleJDG}) and whose cohomology was already computed by G\"ottsche and others
(see \cite{Gottschebook}). We explain how to compute cohomology with our methods in this case and we do explicit computation for the case $n=4$, i.e. $X$ is a $6$-dimensional Kummer manifold (a computation for the case $n=3$ is done in the last section).

Finally in section \ref{general} we consider a $4$-dimensional abelian variety $A$ and we  prescribe the 
action of $G$ only on the tangent space at the unit, i.e we fix
the complex representation $\rho _\C: G \ra GL(\a)$ and we do not ask  a priori that it comes from an integral representation $\rho_\Z$. Moreover we consider three special types of groups and representations coming from a recent theorem which characterizes a class of symplectic singularities admitting a (local) symplectic resolution: see 
section \ref{sympl} where the results in \cite{Bell} and also in \cite{GiKa}, \cite{Lehn-Sorger} are summarized.
We apply the Kummer construction in this context and in one case we prove it can not lead to a global crepant resolution.

\section{Notation and preliminaries}\label{trivialities}
Throughout the paper we will use the set up introduced above: $G$ is a
finite group with representation $\rho_\Z$ in $GL(r,\Z)$; any
extension of $\rho$ will be denoted by a subscript, i.e.~$\rho_\C$
denotes extension of $\rho$ over $\C$. Sometimes we will skip the
subscript if the context is clear. Next, $A$ is a complex abelian
variety of dimension $d$ and $G$ acts on $A^r$ via $\rho_A$. By $Y=A^r/G$
or $A^r/\rho_A$ we will denote the quotient, which we understand as
the space of orbits of the action $\rho_A$ and by $\pi_G: A^r\ra A^r/G$
the quotient morphism.  

Moreover $f: X\ra Y=A^r/G$ is assumed to be a
crepant resolution of the quotient. That is, $X$ is smooth and
projective, $f$ is birational and $K_X=f^*K_Y$.

\smallskip
We denote by $\Z_n$ a cyclic
group of order $n$, by $S_n$ the symmetric group in $n$ letters, or
group of permutations of $n$ elements, and by $D_{2n}$ a dihedral group of
order $2n$, i.e.~semidirect product $D_{2n}=\Z_{n}\rtimes\Z_2$. 

For a group element $g\in G$ by $\langle g\rangle$ we denote the subgroup
generated by $g$.  The normalizer of a subgroup $H<G$ is denoted by
$N_G(H)$, while $W_G(H)$ stands for the quotient group $N_G(H)/H$. By
$[H]_G$ we denote the conjugacy class of $H$ in $G$, that is the set
of subgroups $\{gHg^{-1}<G: g\in G\}$. We skip the subscript whenever
the group in which the above objects are defined is clear from the
context. The partially ordered set of conjugacy classes of subgroups
of $G$ will be denoted by $\cC(G)$ with the partial order denoted by
$\prec$. Recall that the cardinality $\#[H]_G$ is equal to the index
$[G:N_G(H)]$ and for $H'\in [H]_G$ we have $W_G(H')=W_G(H)$, so this
group will be denoted by $W([H]_G)$ and called a Weyl group of $H$ in
$G$. The conjugacy class of an element $h\in G$ is the set
$[h]_G=\{ghg^{-1}: g\in G\}$, note that usually $[\langle h\rangle]\ne
[h]$.

Let $G$ be a group acting on a set $B$ and $H<G$ a subgroup;  by $B^H$
(or $Fix(H)$) we denote the subset of points of $B$ fixed by
$H$ while $B^H_0\subset B^H$ is the set of points whose isotropy (or
stabilizer) is exactly $H$.  Clearly, $B^H\setminus B^H_0$ consists of
points whose isotropy is bigger than $H$. We will use repeatedly the
fact that if $H'=gHg^{-1}$ then $B^{H'}=gB^H$. In particular, the action of $G$ 
defines an action of $N(H)$ and of $W(H)$
on $B^H$ which is free on $B^H_0$.

\smallskip
Let $\a$ be the tangent space of an abelian variety $A$ at identity, i.e. the Lie algebra of
holomorphic vector fields tangent to $A$, and let  $exp: \a\ra A$ be the
exponential map.  The induced action $\rho_\a=\rho_\Z\otimes_\Z\a:
G\ra GL(\a^r)$ splits into $d$ copies of complexified representation
$\rho_\C=\rho_\Z\otimes_\Z\C$. The representation $\rho_\a$ is called the tangent action
and it is in fact  tangent
to $\rho_A$: for $g\in G$ and $p\in A^r$ the derivative of the map
$\rho_A(g): A^r\ra A^r$ at $p$ is $\rho_\a(g): T_{p} A^r =\a^r
\ra T_{g(p)} A^r=\a^r$.  Moreover, if $exp^r: \a^r\ra A^r$ is a natural
extension of $exp$ then for every $g\in G$ and $v\in \a^r$
$$exp^r(\rho_\a(g)(v))=\rho_A(g)(exp^r(v))$$

Let $p\in A^r$ be a point with non-trivial isotropy (or
stabilizer) group $G_p$. In analytic or \'etale topology the action of
$G_p$ around $p$ is equivalent to the induced action
${\rho_\a}_{| G_p} = {d\cdot\rho_\C}_{|G_p}$ in a neighborhood of $0$ in  $T_pA^r \iso \C^{dr}$;
this is because ${\rho_\a}_{| G_p}$ is just
the tangent action of $G_p$ at the tangent space $T_pA$.

\smallskip
More generally if we have an endomorphism $g: A \ra A$  with $g(0) = 0$
it has an associated linear analytic representation  $\eta_{an}(g) : \C^d \ra \C^d$, which is the corresponding 
analytic endomorphism on the universal covering or the derivative map at the origin, (see for instance Proposition 
1.2.1 in  \cite{Birkenhake-Lange}).
The following is the well-known (holomorphic) Lefschetz fixed point formula for
the case of complex tori, as for instance on \cite{Birkenhake-Lange}.

\begin{theorem}
\label{Lefschetz}
Let $g: A \ra A$ be an endomorphism with $g(0) = 0$ and let $\eta_{an}(g) : \C^d \ra \C^d$ as above.  The closed analytic subvariety of $A$ consisting
of the fixed point of $g$, denoted by $Fix(g)$, has dimension equal to the multiplicity of $1$ as an eigenvalue of
$\eta_{an}(g)$. If it is zero dimensional then $|Fix(g)| = |det(1 - \eta_{an}(g))|^2$. 
\end{theorem}

\subsection{Groups and representations}
\label{representations}

Integral representations of finite groups are fairly well understood; in general we will
refer to \cite{Newman} and  \cite[Ch.~XI]{CurtisReiner} or, for an elementary overview,  to \cite{FinGpsSurveyAMM}. 

For a finite group $G$ we consider the  ring $R(G)$ of complex
representations of $G$. By $d\cdot\rho$ we denote the sum of $d$ copies of
the representation $\rho$ while by $\rho^{\otimes d}$ and $\rho^{\wedge d}$ we
denote $d$-th tensor and, respectively, alternating power of $\rho$.
Complex representations of rank 1 will be denoted by roots
of unity.  

We have an additive map $\mu_0: R(G)\ra\Z$ which to a
representation $\rho$ assigns the rank of its maximal trivial
subrepresentation.  If $\rho_\Z\in GL(r,\Z)$ is an integral
representation of rank $r$ with complexification $\rho_\C$ then
$r_0=\mu_0(\rho_\C)$ is the rank of the maximal trivial
subrepresentation of $\rho_\Z$ as well,
c.f.~\cite[Thm.~73.9]{CurtisReiner}. Indeed, set
$\Lambda(\rho_\Z)=\{v\in\Z^r: \forall g\in G\ \ \rho_\Z(g)(v)=v\}$
then $\Lambda(\rho_\Z)$ is a subgroup of $\Z^r$ whose extension to
$\C$ is the maximal trivial subrepresentation of $\rho_\C$. Moreover,
we note that $\Lambda(\rho_\Z)$ is a saturated , i.e.~if $n\cdot
v\in\Lambda(\rho_\Z)$ then $v\in\Lambda(\rho_\Z)$. Thus the quotient
$\Z^r/\Lambda(\rho_\Z)$ has no torsions, $\rho_\Z$ descends to a representation
$\eta_\Z: G\ra GL(\Z^r/\Lambda(\rho_\Z))=GL(r-r_0,\Z)$ and we have a
$G$-equivariant exact sequence, with trivial action on the kernel,
\begin{eq}
0\ra\Lambda(\rho_\Z)\iso\Z^{r_0}\ra\Z^r\ra\Z^r/\Lambda(\rho_\Z)\iso\Z^{r-r_0}\ra 0
\label{exact-sequence}\end{eq}
We will say that $\rho_\Z$ is a pullback of $\eta_\Z$.

\smallskip
Let $S_{r+1}$ be the symmetric group; a {\em natural} representation $\nu_\Z: S_{r+1}\ra GL(r+1,\Z)$ is defined
by permuting coordinates. That is, for a point
$(e_0,e_1,\dots,e_r)\in\Z^r$ (the choice of coordinates will become
clear later) and a permutation $\sigma\in S_{r+1}$ we set
$$\nu_\Z(\sigma)(e_0,\dots,e_r)=(e_{\sigma(0)},\dots,e_{\sigma(r)})$$
In other words, $\nu_\Z(S_{r+1})$ is generated by elementary matrices
$E_{ij}$ responsible for transposing $i$-th and $j$-th vectors of the
chosen basis. Clearly, $(E_{ij}^T)^{-1}=E_{ij}$, where $^T$ denotes
transposition of the matrix, so this representation is isomorphic to
its dual. The natural representation of $S_{r+1}$ contains a fixed
subspace $e_0=\cdots=e_r$ and thus, as above in \ref{exact-sequence},
$\nu_\Z$ is a pull-back of a {\em quotient} representation $\eta_\Z:
S_{r+1}\ra GL(r,\Z)$. On the other hand, we have a $S_{r+1}$-invariant
subspace $e_0+\cdots+e_r=0$ which defines a representation $\rho_\Z:
S_{r+1}\ra GL(r,\Z)$ which we call {\em standard}
representation. Hence the standard representation, being the kernel of
$(e_0,\dots,e_r)\mapsto e_0+\cdots+e_r$, is dual (as $\Z$-module) to
quotient representation. In addition, the $S_{r+1}$-equivariant
composition of inclusion and quotient
$\Z^r\hookrightarrow\Z^{r+1}\ra\Z^r$ makes $\rho_\Z$ a
subrepresentation of $\eta_\Z$ with torsion cokernel and induces an
isomorphism for complexifications $\rho_\C=\eta_\C$ which implies
splitting of complex representation $\nu_\C={\bf 1}_\C+\rho_\C$.
However, $\eta_\Z$ is not conjugate to $\rho_\Z$ in $GL(r,\Z)$,
c.f.~\cite[p.~505]{CurtisReiner}.

\smallskip
Let $G_{r,m} := \Z_m^r\rtimes S_r$ be the semidirect product, where $S_r$ acts on $\Z_m^r$
by permutations. 
We have a natural action of $G_r$ on $\C^r$, namely the group $\Z_m^r$ acts on $\C^r$ diagonally 
and $S_r$ by permutations of the coordinates. We have
a sequence of quotients
\begin{eq}\C^r\lra (\C/\Z_m)^r\lra (\C/\Z_m)^r/S_r= \C^r/G_{r,m}.\end{eq}

\smallskip
Let us finally describe a special group, the binary tetrahedral group $T$, 
and its (complex) representations; we follows the description given in  \cite{Lehn-Sorger}. 
$T$  is the preimage of the symmetric
group of a regular tetrahedron, $T_0$, via the natural map $SU(2) \ra SO(3)$. As a subgroup of $SU(2)$ it is generated by the elements
$$ I= \left(\begin{array}{rr}i&0\\0&-i\end{array}\right) \hbox{     \ \ and \ \ \    }  \tau = -{1\over 2}\left(\begin{array}{rr}1+i&-1+i\\1+i&1-i\end{array}\right).$$
It can also
be described as the semidirect product $Q_8 \rtimes \Z_3$ of the quaternion group 
$Q_8 =\{\pm 1, \pm I, \pm J, \pm K\}$ and the cyclic group
of order $3$.
This group has $7$ irreducible complex representations, we are interested in the two dimensional ones. 
The standard arising from the embedding $T \subset SU(2)$, $\rho_0$, and other two being $\rho_j := \rho_0 \otimes \C_j$, for $j =1,2$, where $\C_j$ is the $1$-dimensional representation given by the multiplication by a third root of unity $\epsilon _3^j$. The last two are dual to each other.

\subsection{Crepant resolutions.}\label{tools}
Working with Kummer construction requires the 
use of crepant resolutions of quotient Gorenstein
singularities.  From dimension 4 on they are far from being understood.
The following is a list of references which we found useful,  it is of course not exaustive: ~\cite{ReidYPG}, \cite{DaisHenkZiegler}, \cite{ReidBourbaki}
for crepant resolution of canonical
singularities and \cite{GiKa}, \cite{FuSurvey} for the symplectic case.

In what follows we will use known facts about resolving 2 dimensional singularities,
for which we refer to ~\cite{BarthPetersVandeVen}. We will also use the following elementary
observation about constructing consecutive resolutions.

\begin{lemma}
\label{solvableresolution}
Consider finite subgroups $H<G<SL(n,\C)$, where $H\triangleleft G$ is
a normal subgroups of $G$. Let $X_H\ra\C^n/H$ be a resolution of the
quotient singularities which is $G/H$-equivariant.  That is, the
quotient group $G/H$ acts on $\C^n/H$ and assume that this
action lifts up to $X_H$. If $X_G \ra X_H/(G/H)$ is a
resolution then the composition with the induced map $X_G\ra
X_H/(G/H)\ra\C^n/G$ is a resolution of the quotient singularity.  If
both intermediate resolutions are crepant then the resulting
resolution is crepant as well.
\end{lemma}

\smallskip
Finally one of our tools will be McKay correspondence which allows to describe
the structure of a crepant resolution of a quotient singularity in
terms of the conjugacy classes, or representations, of the groups
itself, see \cite{ReidBourbaki}, for an exposition in this regard, as well as \cite{KaledinInv},  \cite{GiKa}.

\subsection{Symplectic case}
\label{sympl}

Let $G$ be a finite group with a complex representation $\rho_\C : G \ra GL(V)$. Then $G$ has a symplectic representation $\rho\oplus \rho^*: G \ra Sp(V \oplus V^*)$, where $\rho*$ is the dual representation: the symplectic form preserved is given by the identity
in $V\otimes V^*$.   Moreover if the representation $\rho_\C$ preserves a non degenerate symmetric $2$-form on $V$  then there is a 
$G$-equivariant isomorphism $V \iso V^*$; in this case therefore $2\rho_\C: G \ra Sp(V\oplus V)$ is a symplectic representation.
A primary example is the case when $G$ is a Weyl group  acting on the
lattice of roots of a simple Lie algebra: the action of $G$ preserves
the Killing form (see \cite{BourbakiLie}).

Let $\rho: G \ra Sp(V)$ be a complex symplectic representation. The symplectic form $\sigma$ on $V$ descends to a symplectic form $\hat \sigma$ on the regular part of $V/G$. A proper morphism $f: X \ra V/G$ is a symplectic resolution if $X$ is smooth and $f^*(\hat \sigma)$ extends to a symplectic form on $X$.

More generally a normal variety $Y$ is called a symplectic variety if its smooth part admits a holomorphic symplectic form $\sigma$ whose pull back to any resolution $f: X\ra Y$ extends to a holomorphic $2$-form on $X$.
If this extended holomorphic two form is a symplectic form then $f$ is called a symplectic resolution.
Note that if $Y$ is a symplectic variety and $f: X \ra Y$ is a resolution, then $f$ is symplectic if and only if it is crepant (see for instance proposition 1.3 in \cite{FuSurvey}).

Symplectic resolutions are very rare and they have been considered by a number of people; the following two results give necessary conditions for the existence of a symplectic resolution for quotient singularities. They were proved in  \cite{VerbitskyAsian} and \cite{KaledinSelecta}; see the sections 3 and 4 of the survey  \cite{FuSurvey}
also for appropriate references.

\begin{proposition} Let $X$ be a smooth irreducible symplectic variety and $G$ a finite group of
symplectic automorphisms on X. Assume that $Y = X/G$ admits a
symplectic resolution, then the subvariety $F = \bigcup_{g\not= 1} Fix(g) \subset X$ is either empty or of pure codimension $2$ in $X$. In particular if  $Y$ has an isolated symplectic singularity, then it admits a
symplectic resolution only it is of dimension $2$. 
\label{semism}
\end{proposition}

\begin{proposition} Let $V = \C^{2n}$ and $G$ a finite group of
symplectic automorphisms on $V$, i.e. $\rho_\C : G \ra Sp(\C^{2n})$. Assume that $Y = V/G$ admits a
symplectic resolution, then $G$ is generated by symplectic reflections, i.e. elements $g$ 
such that $codim(Fix(g)) =2$.
In the special case in which $\rho_\C= \eta_\C \oplus \eta_\C^*: G \ra Sp(\C^n \oplus {\C^n}^*)$ is a sum of a complex representation $\eta_\C$ and its dual,  then $V \oplus V^*/G$
has a symplectic resolution if $\eta_\C: G \ra GL(\C^n)$ is generated by complex reflections (i.e. elements $g$ 
such that $codim(Fix(g)) =1)$).
\end{proposition}

Recently the following necessary and sufficient condition has been proved in  \cite{GiKa} and \cite{Bell}; see also \cite{Lehn-Sorger}.
\begin{proposition} 
\label{BLS}
Let $G$ be a finite group, $\rho_\C : G \ra GL(\C^n)$ an irreducible complex representation and assume that 
$\rho_\C(G)$ is  generated by complex reflections. Then $V \oplus V^*/G$
has a symplectic resolution if and only if $(G, \rho_\C)$ is one of the following:

1) $S_{n+1}$ and $\rho_\C$ is the standard representation.

2)  $G_{n,m} = \Z_m^n\rtimes S_n$  and $\rho_\C$ is the natural representation described in the previous section.

3)  $Q_8 \rtimes \Z_3$, the binary tetrahedral group $T$, $n = 2$ and $\rho_\C$ is the representation $\rho_1$ described in the previous section.

\end{proposition}

The first case corresponds to the Weyl groups of type $A$,
namely the Weyl groups of the Lie algebra $\a_n$. The second case, for
$m =2$, corresponds to the Weil  groups of type $B$ and $C$ (see \cite{BourbakiLie}, 
Chapter VI, Tables I, II and III).
Note that in the first case and in the second when $m=2$ the representation is integral.

A local symplectic resolution is obtained in 1) and 2) via 
Hilbert schemes: namely for a smooth surface $S$ the Hilbert scheme $Hilb^n(S)$ provides
a crepant resolution $Hilb^n(S) \ra Sym^n(S)$. 
In 1)  consider the null-fiber of the morphism which is the composition  

$$ Hilb^{n+1}(\C^2) \ra Sym^{n+1}(\C^2)\ra \C^2 $$

where  $s: Sym^{n+1}(\C^2) \ra \C^2$,
is the the summation $\Sigma_{i= 0,..., n} e_i^j$ for $j=1,2$; note that $s^{-1} (0) = \C^{2n}/S_n$. 

In  2) consider first a minimal resolution of the $A_{m-1}$ singularity $\C^2/\Z_m$,
namely ${\widehat {\C^2/\Z_m}} \ra \C^2/\Z_m$, and then the composition 
$$Hilb^n(\widehat{\C^2/\Z_m}) \ra Sym^n(\widehat{\C^2/\Z_m}) \ra Sym^n({\C^2/\Z_m}).$$

An explicit local resolution for the third case has been given recently in \cite{Lehn-Sorger}.

\smallskip
The Kummer construction applied to an abelian surface for the group and the representation in 1) of Proposition \ref{BLS} gives a
generalized Kummer variety $Kum^n(A)$, as constructed by Beauville, see
\cite{BeauvilleJDG} and also \cite{Fujiki}. 
A global resolution  $f: X\ra Y=A^n/S_{n+1}$ is
obtained, as in the local case, considering  the null-fiber of the composition $Hilb^{n+1}(A)\ra
A^{n+1}/S_{n+1} \ra A$. The first map is  the Hilbert-to-Chow map where the action of $S_{n+1}$ on $A^{n+1}$ is by permutation of the  coordinates $(e_0, \dots, e_n)$, the quotient
is interpreted as the Chow variety of $A$; the second is the summation $\Sigma_{i= 0, ...,n} e_i $ .

In the case 2) of Proposition\ref{BLS}, with $m=2$ (the case of integral representation), the Kummer
construction applied to an Abelian surface gives the other series of symplectic manifolds considered by Beauville, see
\cite{BeauvilleJDG} and also  \cite{Fujiki}. 

Namely the group $\Z^n_2$ acts on $A^n$ diagonally so we have
a sequence of quotients
\begin{eq}A^n\lra (A/\Z_2)^n\lra (A/\Z_2)^n/S_n\end{eq}
and since the first quotient can be desingularized as $(Kum^1(A))^n$
then the latter quotient has a natural desingularization as
$Hilb^n(Kum^1(A))$, as follows from lemma \ref{solvableresolution}.

In the last section we will prove that a (generalized) Kummer construction in the case 3) of Proposition \ref{BLS}
cannot have a global crepant resolution.

\section{Computing cohomology}

In this paper we will calculate the De Rham
cohomology, or the Betti numbers $b_i(X)={\rm dim}_\C H^i(X,\C)$, for some
varieties which are the results of Kummer constructions in the low dimensional cases.

This is a two step process which consists of: (1) calculating
cohomology of $Y=A^r/G$ and (2) calculating the contribution coming
from resolution $X\ra Y$.

Recall that one can assign to any complex algebraic variety $X$, not necessarily smooth, or compact, or irreducible, a virtual Poincar\'e polynomial, $P_X(t)$, with the following properties. For a compact manifold $X$ of complex dimension $n$ the polynomial is defined as the standard Poincar\'e polynomial
$$P_X(t)=\sum_{i=0}^{2n} b_i(X)\,t^i\in\Z[t],$$
where $t$ is a formal variable and
$b_i(X)=dimH^i_{DR}(X)$ are the Betti numbers. Moreover if $Y$ is a closed algebraic subset of $X$ and $U:= X \setminus Y$ then 
$$P_X(t) = P_Y(t) + P_U(t).$$
For further details we refer to \cite[4.5]{FultonToric} and \cite[2]{Totaro}.  
We remark that
the virtual Poincar\'e  is actually the standard Poincar\'e polynomial 
also if $X$ is compact and has quotient singularities, see  \cite[p.~94]{FultonToric}.

\subsection{Quotients}\label{cohomology1}
Computing cohomology of the quotient is pretty straightforward: we
look at $H^i(A^r,\C)$ or $H^{pq}(A^r)$ on which $G$ acts via
representation $\left(2d\cdot\rho_\C\right)^{\wedge i}$ or
$(d\cdot\rho_\C)^{\wedge p}\otimes(d\cdot\rho_\C)^{\wedge q}$,
respectively. Note that $\overline\rho_\C=\rho_\C$, because
$\rho$ is real. By looking at the identity component of this
representation one determines the dimension of the space of
$G$-invariant forms which subsequently can be used to get the
information about the cohomology of $Y=A^r/G$, see
e.g.~\cite[Ch.~III]{Bredon}. We note that products of representations
of finite groups can be calculated in a standard way by looking at
their characters, see e.g.~\cite[Part I]{FultonHarris}.

Given an action of a group $G$ on a variety $Z$, in order to formulate the result
 in terms of Poincar\'e polynomial, we define a  $G$-Poincar\'e polynomial, $P_{Z,G}(t)\in R(G)[t]$, 
 whose coefficient at $t^i$ is equal to the representation  
 of the induced $G$-action on the vector space $H^i(Z,\C)$. 
 In particular, in our set-up
\begin{eq}
P_{A^r,G}(t)=\sum_{i=0}^{2rd}(2d\cdot\rho_\C)^{\wedge i}\cdot t^i ;
\label{GPoincare-torus}\end{eq}
we will denote this polynomial by $(1+t)^{2d\rho}$.

\begin{lemma}\label{quotient-polynomial}
For $Y=A^r/\rho_A$ we have $P_Y(t)=\mu_0((1+t)^{2d\rho})$ where $\mu_0:
R(G)[t]\ra \Z[t]$ is the reduction of coefficients via $\mu_0: R(G)\ra\Z$.
\end{lemma}

\subsection{Resolution.}
For understanding the cohomology of the resolution $X\ra Y$ we will
write the quotient $Y=A^r/G$ as a disjoint sum of locally closed sets
(strata) $Y([H])$ consisting of orbits of points whose isotropy is in
the conjugacy class of a subgroup $H<G$. The calculation is somehow in
the spirit of \cite{Hirzebruch}. Over $Y([H])$ the singularities of
$Y$ will be locally quotients of $\C^{rd}$ by action of $H$. Thus, by
taking inverse images of sets $Y([H])$ we will produce a decomposition
of $X$ into a disjoint sum of locally closed sets $X([H])$ such that the
restriction $X([H])\ra Y([H])$ will be a locally trivial fiber bundle
with a fiber $F([H])$ depending on the resolution of the $H$-quotient
singularity. Now the cohomology of $X$ will be computed by looking at
each of $X([H])$ and using virtual Poincar\'e polynomial for each of
them.

The set $Y([H])$ may be disconnected, as for instance in the case of a Kummer
surface construction. However, as already noted, for every $y\in
Y([H])$ the singularity of $Y$ in a neighborhood of $y$ is of type
$\C^{rd}/d\rho_\C(H)$. That is, given $p\in A^r$ in the orbit
represented by $y\in Y([H])$ with isotropy $G_p=H$, there is a map
$T_pA^r/\rho_\a(H)=\a^r/\rho_\a(H)\ra Y$ which is an isomorphism of 
analytic neighborhoods of the $0$ orbit and of $y$.  Indeed, consider
evaluation of vector fields at $p$, that is $exp_p:\a^r=T_pA^r\ra
A^r$, where $\exp_p(0)=p$; it is $H$-equivariant and thus it defines
a map of quotients $T_pA^r/\rho_\a(H)\ra A^r/\rho_A(H)$. Compose
it with the natural map of spaces of orbits $\pi_{G/H}: A^r/H\ra
A^r/G$, coming from the inclusion $H<G$. The first of these morphism
is an isomorphisms of analytic neighborhoods of $0$ and $p$. On the
other hand, we can choose an analytic open neighborhood $U$ of $p$
such that $gU=U$ for $g\in H$ and $gU\cap U=\emptyset$ for $g\not\in
H$. Thus the orbits of $G$ and $H$ restricted to $U$ coincide, that is
$\pi_{G|U}=\pi_{H|U}$, for the respective orbit class maps. Hence
$\pi_{G/H}$ is bijective and, by normality of quotients, an
isomorphism over the respective neighborhoods of $\pi_H(p)$ and $y$.

By these arguments, there exists an open analytic neighborhood
$V$ of $\pi_H((A^r)^H_0)$ such $\pi_{G/H}$ restricted to $V$ is a
local isomorhism onto an open neighborhood of $Y([H])$.

Now let $r_0=\mu_0(\rho_{|H})$ be the rank of the maximal trivial
subrepresentation of $(\rho_\C)_{|H}$ so that, as in formula
\ref{exact-sequence}, ${(\rho_\Z)}_{| H}$ is a pull-back of $\eta_H:
H\ra SL(r_H,\Z)$ of rank $r_H$, with $r_H+r_0=r$, and $\eta$ has no
non-trivial fixed point. Accordingly, after extending to $A$, we have
a $H$-equivariant sequence of abelian varieties \begin{eq}
A^{r_0}\hookrightarrow A^r \longrightarrow A^{r_H}\end{eq} where the
action of $H$ on $A^{r_H}$ has only a finite number of fixed
points. Thus $(A^r)^H$ is a union of (affine) abelian subvarieties of
$A^r$ of dimension $dr_0$.

We fix a resolution of the quotient singularity $\C^{dr_H}/d\eta_\C$
with the special (central) fiber $F(H)$. Then it determines a product
resolution of a neighborhood of a component of $(A^r)^H$ in
$A^r/H$. Indeed, as follows from the preceding discussion, any such
component has a neighborhood isomorphic to a neighborhood of
$A^{r_0}\times\{0\}$ in $A^{r_0}\times \C^{dr_H}/d\eta_\C$, hence it
admits a product resolution. We will consider resolutions $X\ra Y$
which are locally product in the following sense.

\begin{definition}\label{locally-product-resolution}
Let $f: X\ra Y=A^r/G$ be a resolution of singularities. We say that it
is {\sl a locally product} if for every $H<G$ and every irreducible component
$K$ of $Y([H])$ there exists an open analytic neighborhood $U\subset
Y$ of $K$ such that the pull-back via $\pi_{G/H}$ of the resolution
$f^{-1}(U)\ra U$ over an open subset of $A^r/H$ is analytically
equivalent to a product resolution with the special fiber $F([H])$
depending only on the conjugacy class of $H$.
\end{definition}

The next result provides a description of both $Y([H])$ and $X([H])$
in terms of $(A^r)^H_0$ and the group $W(H)$. 

\begin{lemma}\label{structure-of-strata}
Let $f: X\ra Y=A^r/G$ be a locally product resolution of
singularities. Then $f_{|X([H])}: X([H])\ra Y([H])$ is an \'etale fiber
bundle whose fiber $F([H])$ is isomorphic to the special fiber of a
resolution of the quotient singularity $\C^{dr_H}/H$. Moreover, let
$\overline{Y([H])}\subset Y$ denote the closure of $Y([H])$ in
$Y$ and $\widehat{Y([H])}\ra \overline{Y([H])}$ be its
normalization.  Then the following holds
\begin{itemize}
\item The action of $N(H)$ determines an action of $W(H)$ on
$\overline{(A^r)^H_0}$  and the morphism $\overline{(A^r)^H_0}\ra\widehat{Y([H])}$ is the
quotient by $W(H)$.
\item The action of $W([H])$ on $\overline{(A^r)^H_0}$ lifts to the
product $\overline{(A^r)^H_0}\times F([H])$
in such a way that there is a commutative diagram
$$\xymatrix{
\overline{(A^r)^H_0}\times F([H])\ar[r]\ar[d]&
\left(\overline{(A^r)^H_0}\times F([H])\right)/W([H])\ar[d]&
X([H])\ar[l]\ar[d]\\
\overline{(A^r)^H_0}\ar[r]&\widehat{Y([H])}&Y([H])\ar[l]
}$$
where the horizontal arrows on the left hand side are quotient maps
while these on the right hand side are inclusions onto open subsets
\end{itemize}
\end{lemma}
\begin{proof}
Take the quotient map $A^r\ra Y=A^r/G\supset Y([H])$ and consider
inverse image of $Y([H])$.  The inverse image decomposes into disjoint
sets $(A^r)^{H'}_0$, depending on the isotropy class $H'\in[H]$. The
normalizer $N(H)$ acts on the set of points whose isotropy is $H$,
i.e.~$(A^r)^H_0$, and this determines a free action of $W(H)$ 
on this set of points. Take factorization of the quotient map
$\pi_G$ into $A^r\ra A^r/W(H)\ra A^r/G$ which gives a regular birational
map $(A^r)^H_0/W(H)\ra \overline{Y([H])}$.  This proves the central
statements of the lemma. The rest follows because of our assumption
regarding the resolution.
\end{proof}

We note that the map $(A^r)^H_0\times F([H])\ra X([H])$ usually does
not extend, so that $\overline{(A^r)^H_0}\times F([H])/W([H])$
is not the normalization of the closure of $X([H])$. 

We will use the preceding lemma to calculate the Poincar\'e polynomial of
both $Y([H])$ and $X([H])$. In fact, the following is how one computes
the cohomology of $\widehat{Y([H])}$ and $\overline{(A^r)^H_0}\times
F([H])/W([H])$. Let $K\subset Y([H])$ be an irreducible
component whose normalized closure we denote by $\widehat{K}$. Then
$\widehat{K}\iso A_K/W_K$ where $W_K<W(H)$ is the subgroup which
preserves $A_K\iso A^{r_0}$, a component of the closure of $(A^r)^H_0$
which dominates $K$. Thus, as in \ref{GPoincare-torus}, we can write
\begin{eq}
P_{A_K,W_K}(t)=\sum_{i=0}^{2dr_0}(2d\cdot\eta_K)^{\wedge
i}\cdot t^i= (1+t)^{2d\eta_K}
\label{GPoincare-Y}
\end{eq}
where $\eta_K: W_K\ra GL(r_K,\C)$ is a representation of $W_K$ induced
from $\rho_\C$. That is, the group $N(H)$, and thus $W(H)$, acts on
the fixed point space of $(\rho_\C)_{|H}$, as in \ref{exact-sequence},
hence it yields an action of $W_K<W(H)$. Therefore, as in
\ref{quotient-polynomial}, we get
$P_{\widehat{K}}=\mu_0(P_{A_K,W_K})$.

\bigskip Now, recall that the McKay correspondence postulates a
canonical relation of conjugacy classes of elements in a group $H$
with cohomology or homology of a crepant resolution of its quotient
singularity, see e.g.~\cite{ReidBourbaki}, 
\cite{BertinMarkushevich},  \cite{KaledinInv} and \cite{GiKa}. Thus, if the McKay correspondence holds for
the fixed resolution of $\C^{r_Hd}/H$ then we can use it to understand
the action of $W(H)$ on cohomology of $\overline{(A^r)_0^H}\times
F(H)$. Indeed, the group $W(H)$ acts on the cohomology of
$F(H)$ as $W(H)$ acts on the conjugacy classes of $H$. So, in the
situation introduced in the previous paragraph, the $W_K$-Poincar\'e
polynomial $P_{F(H),W_K}$ is determined by the adjoint action of $W_K$
on conjugacy classes of elements in $H$, which is $w([h]_H)\mapsto
[whw^{-1}]_H$, where $w\in N(H)$ represents an element of
$W(H)$ and $h\in H$. Thus, whenever the representation
$\rho$ is fixed, we will simply write $P_{H,W_K}$ instead of
$P_{F(H),W_K}$. We conclude
\begin{eq}
P_{(A_K\times F(H)), W_K}=(1+t)^{2d\eta_K}\cdot P_{H,W_K}
\label{PoincareWF/G}
\end{eq} 
However, deriving from it the virtual Poincar\'e polynomial for $X([H])$
requires understanding lower dimensional strata, that is the quotient
$(A_K\times F(H))/W_K$ over the difference $\overline{Y([H])}\setminus
Y([H])$. To this end, let $H'> H$ be a subgroup and $K'\subset
\overline{K}$ an irreducible component of $Y([H'])$. By $W_{K;K'}<W_K$
we denote the subgroup of $W_K$ which preserves $K'$. Then, by
restricting the representations we get $P_{H,W_{K;K'}}$ which is the
$W_{K;K'}$-polynomial describing the action of $W_{K;K'}$ of the
cohomology of the fiber $F(H)$. On the other hand we have an induced
representation $\eta_{K;K'}: W_{K;K'}\ra GL(r_{K'},\C)$ and therefore
the action of $W_{K;K'}$ on the cohomology of the torus $A_{K'}\iso
A^{r_{K'}}$ dominating $K'$ is described by the polynomial 
$(1+t)^{2d\eta_{K;K'}}$, thus
\begin{eq}
P_{(A_{K'}\times F(H)), W_{K;K'}}=(1+t)^{2d\eta_{K;K'}}\cdot P_{H,W_{K;K'}}
\end{eq}

The following statement is a summary of the preceding discussion.

\begin{principle}
Let $X$ be obtained via a Kummer construction as described in the introduction; i.e. $X$ is a crepant resolution
of $A^r/G$ satisfying  a) and b) in the introduction. The Poincar\'e polynomial $P_X\in\Z[t]$ is
described by the following formula
$$
\sum_{[H]\in\cC(G)}\sum_{K\subset Y([H])}
\left[\begin{array}{l}
\mu_0\left(  (1+t)^{2d\eta_K}\cdot P_{H,W_K}\right) - \\ 
\sum_{[H']\succ[H]}
\sum_{K'\subset\overline{K}\cap Y([H'])}
a_{K;K'}\cdot\mu_0\left((1+t)^{2d\eta_{K;K'}}\cdot P_{H,W_{K;K'}}\right)
\end{array}
\right]
$$
where $K$ runs through all irreducible components of $Y([H])$ and $K'$
through all irreducible components of $\overline{K}\cap Y([H'])$ and
$a_{K;K'}$ are numbers which depend on the incidence of the closures
of components of the stratification of $Y$.
\label{principle}
\end{principle}

The data which appears in the above formula is of two types: (1)
depending on the group $G$ and its complex representation
$\rho_\C$ and (2) depending on the integral conjugacy class of the
representation $\rho_\Z$ which determines the geometry of the quotient
$Y$ and its stratification. 
As it is shown in \cite{Donten},  already in dimension three Kummer constructions with the same $\rho_\C$ can have different $\rho_\Z$ and different cohomology.

\section{Building upon elliptic curves: $d=1$}
\label{elliptic}

In the present section $A$ denotes an elliptic curve.  Quotients of
products of elliptic curves by actions of specific groups have been
considered by several people: \cite{Paranjape}, \cite{CynkHulek},
\cite{CynkSchuett}.

\subsection{Special Kummer surfaces: $r=2$}

Let us start with the following easy classical case.
The classification of rank 2 groups whose action give Gorenstein
singularities, known as Du Val singularities, is very well understood,
these are finite subgroups of $SL(2,\C)$, \cite{Durfee15}. On the
other hand there are only 4 types of nontrivial subgroups of
$SL(2,\Z)$, all are cyclic and generated, up to conjugation in
$GL(2,\Z)$, by one of the following matrices
(cf.~\cite[Ch.~IX]{Newman}).
\begin{eq}\begin{array}{cccc}
\left(\begin{array}{rr}-1&0\\0&-1\end{array}\right)&
\left(\begin{array}{rr}0&-1\\1&-1\end{array}\right)&
\left(\begin{array}{rr}0&-1\\1&0\end{array}\right)&
\left(\begin{array}{rr}0&-1\\1&1\end{array}\right)
\end{array}\end{eq}
They generate cyclic groups $\Z_2$, $\Z_3$, $\Z_4$ and $\Z_6$. 

Let us discuss the case of $\rho_\Z: \Z_6\ra SL(2,\Z)$.  In the
following table we list its subgroups, each of them generated by an
element $g$. For each of them we give the number of its fixed points
and the number of singular points whose isotropy is exactly the group
in question. The latter number is obtained by subtracting those fixed
points whose isotropy is bigger and dividing by the cardinality of the
respective orbit of $G$, which is the index of the subgroup in
question. In the last two columns we present the Dynkin diagram of the
special fiber of the minimal resolution of the respective singular
point and its virtual Poincar\'e polynomial.

\begin{tabular}{ccclc}
$g$&$\#$ fix pts& $\#$ sing pts &resolution&Poincar\'e\\
$\left(\begin{array}{rr}0&-1\\ 1&1\end{array}\right)$&1&1&
$\begin{xy} <8pt,0pt>:
(1,0)*={\bullet}="1" ; (2,0)*={\bullet}="2"   **@{-},
"2" ; (3,0)*={\bullet}="3"   **@{-},
"3" ; (4,0)*={\bullet}="4"   **@{-},
"4" ; (5,0)*={\bullet}   **@{-},
\end{xy}$&$1+5t$ \\
$\left(\begin{array}{rr}0&-1\\ 1&-1\end{array}\right)$&9&4&
$\begin{xy} <8pt,0pt>:
(0,0)*={\bullet} ; (1,0)*={\bullet}="1" **@{-},
\end{xy}$&$1+2t$\\
$\left(\begin{array}{rr}-1&0\\
0&-1\end{array}\right)$&16&5&$\bullet$&$1+t$\\
\end{tabular}

On the other hand we note that over the complex number, that is in
$SL(2,\C)$, the representation $\rho$ (or, equivalently, the matrix
generating the image of $\rho$) is equal to diagonal representation
$\epsilon_6+\epsilon_6^5$, where $\epsilon_6$ denotes sixth
primitive root of unity. Thus we compute $\rho\otimes\rho=2\cdot{\bf
1}+ \epsilon_6^2+\epsilon_6^4$ hence the space of invariant $(1,1)$
forms is of dimension 2.  We add to it the contribution of cohomology
coming from resolving singular points of the quotient, as listed
above, to get
$$2 + 1\times 5 + 4\times 2 + 5\times 1= 20$$ which is the dimension
of $H^{11}$ for a K3 surface.

\subsection{Kummer threefolds: $r=3$}
\label{Kummer}
Construction of Calabi-Yau threefolds via quotients have been
considered in e.g.~\cite{CynkHulek}, \cite{CynkSchuett} as well as in
\cite{Oguiso}. We note, however, that the last reference concerns
dividing abelian varieties by an action of translations so that the
quotient map is \'etale.  As for the linear action, which is used in
the present paper, there is a classical book \cite[Ch.~IX]{Newman} which
provides a list of isomorphisms classes of finite subgroups of
$SL(3,\Z)$. Donten, \cite{Donten} classifies noncyclic finite
subgroups of $SL(3,\Z)$ up to conjugacy in $GL(3,\Z)$. The following
proposition summarizes these results.

\begin{proposition}\label{class-of-subgs}
The following are, up to isomorphism, (non-trivial) finite subgroups
of $SL(3,\Z)$:
\begin{itemize}
\item cyclic groups $\Z_a$, of rank $a$, for $a=2,\ 3,\ 4$ and 6,
\item dihedral groups $D_{2a}$, of rank $2a$, for $a=2,\ 3,\ 4$ and 6,
  which have, respectively, $4,\ 3,\ 2$ and 1 conjugacy classes in
  $GL(3,\Z)$
\item the alternating group $A_4$ which has 3 conjugacy classes in
  $GL(3,\Z)$ (e.g.~the tetrahedral group of isometries of
  the tetrahedron),
\item the symmetric group $S_4$ which has 3 conjugacy classes in
  $GL(3,\Z)$ (e.g.~octahedral group of isometries of a cube)
\end{itemize}
\end{proposition}

\begin{lemma}\label{1-dim-fixed-pts}
For a non-identity matrix $M\in SL(3,\Z)$ of finite order the fixed
point set of $M$ is of dimension one.
\end{lemma}

\begin{proof}
The eigenvalues of $M$ are roots of unity and their product is 1 and
at least one of them is a real number, hence equal $\pm 1$. If
$\lambda_1$ and $\lambda_2$ are non-real eigenvalues then, as roots of
a degree 3 real polynomial, they are conjugate hence their product is
1. Thus, one of the eigenvalues is 1 and the eigenspace of 1 is either
of dimension 1 or 3.
\end{proof}

In particular the case of cyclic groups is not allowed since we assume that the fixed point set of $G$ is $\{0\}$. 

\smallskip
Consider the case of $D_4=\Z_2\times\Z_2$, generated by matrices 
$$\begin{array}{ccc}
A_{001}=\left(\begin{array}{rrr}-1&0&0\\0&-1&0\\0&0&1\end{array}\right)&
A_{010}=\left(\begin{array}{rrr}-1&0&0\\0&1&0\\0&0&-1\end{array}\right)&
A_{100}=\left(\begin{array}{rrr}1&0&0\\0&-1&0\\0&0&-1\end{array}\right)
\end{array}$$
In toric terms the quotient map $\C^3\ra\C^3/D_4$ is given by
extending the standard integral lattice by adding generators
$(1/2,1/2,0)$, $(1/2,0,1/2)$ and, subsequently, also
$(0,1/2,1/2)$. It admits two types of resolution.  The first one is
obtained by consecutive blow-up of singularities of $\Z_2$ actions, as
in \ref{solvableresolution}, while the second one is invariant with
respect to permutation of coordinates. They differ by a flop. The
picture below presents fans of both of them as of a section of the
first octant. Namely, we present the respective divisions of the
standard cone (first octant) with vertexes (standard basis) denoted by
$\circ$ and the boundary presented by dotted lines. Next, $\bullet$
denote exceptional divisors of the resolution and solid line segments
stand for the 2 dimensional cones of the resolution.

$$\begin{array}{lcr}
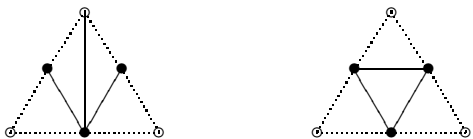
\begin{xy}<14pt,0pt>:
(-2,0)*={\circ}="0" ; (0,0)*={\bullet}="1" **@{.},
"1" ; (2,0)*={\circ}="2" **@{.},
"0" ; (-1,1.73)*={\bullet}="3" **@{.},
"1" ; "3" **@{-},
(1,1.73)*={\bullet}="4" ; "1" **@{-},
"2" ; "4" **@{.},
(0,3.26)*={\circ}="5" ; "3" **@{.},
"4" ; "5" **@{.},
"1" ; "5" **@{-}
\end{xy}
&
\phantom{\begin{xy}<10pt,0pt>:
(0,0)*={},
(-2,1)*={} ; (2,1)*={} **@{-}
\end{xy}}
&
\begin{xy}<14pt,0pt>:
(-2,0)*={\circ}="0" ; (0,0)*={\bullet}="1" **@{.},
"1" ; (2,0)*={\circ}="2" **@{.},
"0" ; (-1,1.73)*={\bullet}="3" **@{.},
"1" ; "3" **@{-},
(1,1.73)*={\bullet}="4" ; "1" **@{-},
"4" ; "3" **@{-},
"2" ; "4" **@{.},
(0,3.26)*={\circ}="5" ; "3" **@{.},
"4" ; "5" **@{.}
\end{xy}
\end{array}$$
We note that the right-hand side resolution is invariant with respect
to the action of permutations of coordinates of the standard cone thus
this resolution satisfies assumptions of lemma
\ref{solvableresolution} with respect to 
$D_4 \triangleleft A_3=D_4\rtimes \Z_3$ or 
$D_4\triangleleft S_4 =D_4\rtimes S_3$.

Let us note that the resolution of singularities is uniquely
defined in codimension 2, hence locally product in the sense of
definition \ref{locally-product-resolution}.

Finally note also that in this $3$-fold case it is not hard to construct a global crepant resolution.
Namely one first constructs a resolution of the $1$-dimensional singular strata: take an element in the isotropy group of a point in this strata, $g$, and consider $Fix(g) = \bigcup C_i$. Let $C_0$ be the component through the origin of 
$A^3$; this is a one dimensional abelian variety which is also a subgroup of $A^3$. Take the quotient $\pi: A^3 \ra A^3/C_0$ and take the induced action of $g$ on the surface $A^3/C_0$: resolve the singularities of the quotient 
$(A^3/C_0)/g$ which are the image of the curves $C_i$ under $\pi$. The lift up of this surface desingularization via $\pi$ will give a desingularization of $A^3/G$ along the strata $C_i$. 
After that one glues the desingularization of the isolated singularities.

\subsection{Cohomology, case of the octahedral group}
\label{octa}
In this section we discuss, as an example, the case of the octahedral group which is just a
representation of $S_4$ into $SL(3,\Z)$. We note that $S_4$ admits other representations in $SL(3,\Z)$
which are not conjugate in $GL(3,\Z)$ to the one which will be considered (the other two can be found in \cite{Donten}).

The following table summarizes information about singularities of
$A^3/G$ in codimension 2. In the first column, we write down conjugacy
classes of non-trivial elements $g$ of this group together with
equations of their fixed points in $A^3$ with coordinates
$(e_1,e_2,e_3)$ (column 2). In each case, because of lemma
\ref{1-dim-fixed-pts}, the fixed point set is a number of elliptic
curves, so in the next column we write the number of components of the
fixed point set.  Next, we write the group generated by this element
$\langle g\rangle$ together with $W(g):=N(\langle g\rangle)/\langle g
\rangle$.  The group $W(g)$ acts on the fixed point set of $g$ and
only in the first case it acts nontrivially on the set of components
while its action on each component (elliptic curve) is an involution
$e\mapsto -e$ hence it has 4 fixed points. Thus, by dividing by action
of $W(g)$, we get the set of singular points of the quotient whose
generic points (in each component) have isotropy $\langle g\rangle$,
we write it in the next column. In the last column we write the
virtual $W(g)$-Poincar\'e polynomial of the fiber of a minimal resolution
of the respective singularity. The polynomial provides the information
about the action of $W(g)$, that is $\epsilon$ is the representation
satisfying $\epsilon^2={\bf 1}$.  We note that since the minimal
resolution in case of surfaces is unique any such resolution will be
locally product in the sense of definition
\ref{locally-product-resolution} hence lemma \ref{structure-of-strata}
can be applied.

\smallskip
\begin{tabular}{ccccccc}
$g$&$Fix(g)$& $\#$ cmpnts&$\langle g\rangle$&
$W(g)$&$\widehat{Y(\langle g\rangle)}$&Poincar\'e\\
$\left(\begin{array}{rrr}-1&0&0\\0&-1&0\\0&0&1\end{array}\right)$&
$\begin{array}{c}2e_1=0\\2e_2=0\end{array}$&
16&$\Z_2$&$\Z_2\times\Z_2$&$6\times\bP^1$&$1+t^2$\\
$\left(\begin{array}{rrr}0&-1&0\\1&0&0\\0&0&1\end{array}\right)$&
$\begin{array}{c}e_1=e_2\\2e_1=0\end{array}$&
4&$\Z_4$&$\Z_2$&$4\times\bP^1$&$1+(2+\epsilon)t^2$\\
$\left(\begin{array}{rrr}0&1&0\\1&0&0\\0&0&-1\end{array}\right)$&
$\begin{array}{c}e_1=e_2\\2e_3=0\end{array}$
&4&$\Z_2$&$\Z_2$&$4\times\bP^1$&$1+t^2$\\
$\left(\begin{array}{rrr}0&0&1\\1&0&0\\0&1&0\end{array}\right)$&
$\begin{array}{c}e_1=e_2\\e_1=e_3\end{array}$&
1&$\Z_3$&$\Z_2$&$1\times\bP^1$&$1+(1+\epsilon)t^2$\\
\end{tabular}

Note that the components of the fixed point sets listed in the above
table meet in a set $\{p\in A^3: 2p=0\}=\{(e_1,e_2,e_3)\in A^3:
2e_1=2e_2=2e_3=0\}$ of cardinality $4^3=64$ which is where are located
points with non-cyclic isotropy groups. We list them in the subsequent
table, together with the virtual Poincar\'e polynomial of the fiber of a
crepant resolution, which by McKay correspondence is related to the
number of conjugacy classes of the respective group.  \par\medskip

\begin{tabular}{ccccc}
subgroup&fixed set in $\{2p=0\}$& $\#$ fixed pts&$\#$ sing
pts&Poincar\'e\\ 
$D_4$&$e_1\ne e_2\ne e_3\ne e_1$&24&4&$1+3t^2$\\ 
$3\times D_8$&$e_i=e_j\ne e_k$, $\{i,j,k\}=\{1,2,3\}$&36&12&$1+4t^2$\\
$G=S_4$&$e_1=e_2=e_3$&4&4&$1+4t^2$\\
\end{tabular}

Now we pass to computing the Poincar\'e polynomials of the respective strata
$Y([H])$. We write our calculation in typewriter type, in the form of
code of {\tt maxima}, \cite{maxima}. We start with the generic strata,
that is $Y([id])$. This is obtained by substracting from the Poincar\'e
polynomial of the quotient $A^3/S_4$ which we calculate by looking at
the invariants of the respective representation, as in \cite[Part
I]{FultonHarris}, the singular locus. The latter consists of 16 copies
of $\bP^1$, each with 4 points removed, and 16 points associated to
non-cyclic subgroups. The result is the polynomial of 3 dimensional
stratum\par {\tt
S3(t):=1+t\^{}2+4*t\^{}3+t\^{}4+t\^{}6-(15*(1+t\^{}2-4)+20);} \par
Next we consider 1-dimensional strata which are associated to cyclic
groups of type $\langle g\rangle$. We use lemma \ref{structure-of-strata}
and formula \ref{PoincareWF/G}. We have already noted how $W(g)$ acts
on $(A^3)^g_0$. On the other hand $W(g)$ acts on the
cohomology of a fiber of the resolution as it does on the conjugacy in
$\langle g\rangle$. Thus, in our case, the only interesting situation
is when $\langle g\rangle$ is either $\Z_3$ or $\Z_4$. In each of
these cases we have to look at the $\Z_2$ representation $\epsilon$, with
$\epsilon^2=1$ and invariant parts of respective polynomials
$(1+2\epsilon t+t^2)(1+(1+\epsilon)t^2)$ and $(1+2\epsilon
t+t^2)(1+(2+\epsilon)t^2)$. From the resulting polynomials one has to
subtract the part related to the fixed points of the action of $W(g)$. The
result is as follows. We write the polynomials associated to the
respective cyclic groups $\Z_2$, $\Z_3$ and $\Z_4$:\par {\tt\obeylines
S12(t):=10*((1+t\^{}2)*(1+t\^{}2)-4*(1+t\^{}2));
S13(t):=((t\^{}4+2*t\^{}3+2*t\^{}2+1)-4*(1+t\^{}2));
S14(t):=4*((2*t\^{}4+2*t\^{}3+3*t\^{}2+1)-4*(1+2*t\^{}2)); } \par
Finally, we consider 0-dimensional strata which, again, we compute
using McKay correspondence: \par{\tt
S0(t):=4*(1+3*t\^{}2)+(12+4)*(1+4*t\^{}2);}\par 

Calculating the sum
{\tt P(t):=S3(t)+S12(t)+S13(t)+S14(t)+S0(t)}
we get:

\begin{proposition}
The Poincar\'e polynomial of a crepant resolution of $A^3/S_4$,
$X \ra A^3/S_4$, is 
$$P_X(t)=t^6+20\,t^4+14\,t^3+20\,t^2+1.$$
\end{proposition}


\section{Building upon abelian surfaces: $d=2$.}
\label{Got}

In this section we consider $A$ of dimension 2, i.e. an abelian surface. 

We will take the group $S_r$ with the standard representation $\rho_{\C}$ and the Kummer construction applied 
to $A$ will give the series of symplectic manifolds $Kum^{(r-1)}$, as noted in section \ref{sympl}. In particular
a global symplectic resolution exists and McKay correspondence holds. Therefore we can compute 
the Poincar\'e polynomial of a crepant resolution of $A^r/S_{r+1}$ with our method, i.e.~using \ref{principle}.
Invariants of Beauville's generalized Kummer manifolds have been dealt
with by G\"ottsche \cite{Gottsche}, G\"ottsche and Soergel
\cite{GottscheSoergel}, Debarre, \cite{Debarre}, Sawon,
\cite{SawonPhD} and Nieper-Wi{\ss}kirchen, \cite{Nieper1},
\cite{Nieper2}. 

The first step is computing the Poincar\'e polynomial of the quotient
$A^r/S_{r+1}$, which is obtained by calculating the invariant
parts of the
representation on $\Lambda^* H^*(A^r,\C)$ which is generated by wedge powers of $4\rho_\C$, see
\ref{quotient-polynomial}.

Next we are to understand the resolution of singularities of
$Y=A^r/S_{r+1}$; for this purpose we split this quotient into strata
related to points with a fixed isotropy group. We recall some standard
facts and definitions regarding the group of permutations, see
e.g.~\cite{JamesKerber}:
\begin{itemize}
\item the conjugacy classes of elements in $S_n$ are determined by their
decomposition into cycles and are described by partitions of $n$, that
is sequences of positive integers whose sum is $n$,
\item $(a_i^{b_i})=(a_1^{b_1},\dots ,a_m^{b_m})$, where $a_1>\cdots>
a_m>0$ and $b_i$ are positive integers, denotes partition
consisting of $b_i$ copies of $a_i$, so that $b_1\cdot
a_1+\cdots+b_m\cdot a_m= n$,
\item for the partition $(a_i^{b_i})=(a_1^{b_1},\dots ,a_r^{b_r})$
define its length equal to $b_1+\cdots+b_m$, in other words this is
the length of the sequence of $a_i$'s, each of them repeated $b_i$
times,
\item the Poincar\'e polynomial of $S_n$ is, by definition,
$P_{S_n}(t)=\sum_0^{n-1}\kappa_i t^{2i}$, where $\kappa_i$ is the number
of partitions of $n$ of length $n-i$,
\item we say that partition $({a_i'}^{b_i'})$ divides (or it is a refinement of) partition
$(a_j^{b_j})$ if the sequence of $a_i'$'s (with repetitions counted by
$b'_i$'s) can be divided into disjoint sequences whose sums yield
$a_j$'s (with repetitions counted by $b_j$'s)
\item given $\sigma\in S_n$, whose decomposition into cycles gives
partition $(a_i^{b_i})$, it determines a Young subgroup $S(\sigma)\iso
S_{a_1}^{\times b_1}\times\cdots\times S_{a_m}^{\times b_m}$,
\cite[Sect.~1.3]{JamesKerber}, the conjugacy class of this group in
$S_n$ will be denoted by $S(a_i^{b_i})$.
\item $N(S(a_i^{b_i}))/S(a_i^{b_i})\iso S_{b_1}\times\cdots\times
S_{b_m}$, c.f.~\cite[Sect.~4.1]{JamesKerber}; we denote this group by
$W(a_i^{b_i})$.
\end{itemize}

Fix coordinates $(e_0,e_1,\dots,e_r)$ on $A^{r+1}$, with $A^r\subset
A^{r+1}$ defined by equation $e_0+e_1+\cdots+e_r=0$.  For a
permutation $\sigma$ such that $[\sigma]_{S_n}=(a_i^{b_i})$, we take
its fixed point set $(A^r)^\sigma$.  Decomposition of $\sigma$ into
cycles gives equations of $(A^r)^\sigma$: for example a cycle
$(0,\dots,m)$ yields equations $e_0=\cdots=e_m$. The same is fixed by
the respective Young group $S(\sigma)$. On the other hand, each
partition defines a closed subset of the quotient $Y=A^r/S_{r+1}$
consisting of orbits of points fixed by the respective conjugacy class
of $S_{r+1}$. Inside this set there is a dense subset consisting of
orbits of points whose stabilizer is in the conjugacy class
$S(a_i^{b_i})$, we will denote it by $Y(a_i^{b_i})=Y(a_1^{b_1}\dots
a_m^{b_m})$. In particular $Y=\overline{Y(1^{r+1})}$ and the set of
fixed points of $S_{r+1}$ is $Y((r+1)^1)$. By $\widehat{Y(a_i^{b_i})}$
we denote the normalization of the closure $\overline{Y(a_i^{b_i})}$.
We have the restriction of the quotient map
$(A^r)^\sigma\ra\widehat{Y(a_i^{b_i})}\ra\overline{Y(a_i^{b_i})}$.

Sets $Y(a_1^{b_1}\dots a_m^{b_m})$ determine a stratification of both
$Y$ and its resolution $X\ra Y$, the inverse image of
$Y(a_1^{b_1}\dots a_m^{b_m})$ will be denoted by $X(a_1^{b_1}\dots
a_m^{b_m})$. Below, we list the facts regarding these sets needed to
compute the cohomology of $X$.

\begin{lemma}\label{Poincare-partition-properties}
In the above set up the following holds, with
$[\sigma]_{S_n}=(a_i^{b_i})$:
\begin{enumerate}
\item the number of fixed points, i.e. $\#Y((r+1)^1)$, is equal
$(r+1)^4$ and, more generally, the number of components of
$Y(a_i^{b_i})$ is $(GCD(a_i))^4$, where $GCD$ stands for greatest
common divisor,
\item the sets $(A^r)^\sigma$ and $Y(a_i^{b_i})$ are of pure dimension
$2(l-1)$, where $l$ is length of $(a_i^{b_i})$,
\item the set $Y({a'}_i^{b'_i})$ is contained in the closure
$\overline{Y(a_j^{b_j})}$ if and only if partition $(a_j^{b_j})$
divides $({a'}_i^{b'_i})$,
\item $\overline{(A^r)^\sigma_0}=(A^r)^\sigma$ and the morphism
$(A^r)^\sigma\ra\widehat{Y(a_i^{b_i})}$ is quotient by $W(a_i^{b_i})$,
\item the resolution $f: X\ra Y$ is locally product as in definition
\ref{locally-product-resolution} and we have the following version of
\ref{structure-of-strata}\begin{list}{$\circ$}{\listparindent=-1cm}
\item the map $X(a_1^{b_1}\dots a_m^{b_m})\ra Y(a_1^{b_1}\dots
a_m^{b_m})$ is \'etale fiber bundle whose fiber $F(a_1^{b_1}\dots
a_m^{b_m})$ is isomorphic to the product $F(a_1)^{\times
b_1}\times\cdots\times F(a_m)^{\times b_m}$ and has Poincar\'e
polynomial equal to $P_{S_{a_1}}^{b_1}\cdots P_{S_{a_m}}^{b_m}$
\item the action of $W(a_i^{b_i})$ lifts to the product
$(A^r)^\sigma\times F(a_i^{b_i})$ with $W(a_i^{b_i})\iso
S_{b_1}\times\cdots\times S_{b_m}$ acting on $F(a_i^{b_i})\iso
F(a_1)^{\times b_1}\times\cdots\times F(a_m)^{\times b_m}$ by
permuting respective factors of the product, that is $S_{b_i}$
permuting factors of $F(a_i)^{\times b_i}$,
\item there is a commutative diagram
$$\xymatrix{
(A^r)^\sigma\times F(a_i^{b_i})\ar[r]\ar[d]&
\left((A^r)^\sigma\times F(a_i^{b_i})\right)/W(a_i^{b_i})\ar[d]&
X(a_i^{b_i})\ar[l]\ar[d]\\
(A^r)^\sigma\ar[r]&\widehat{Y(a_i^{b_i})}&Y(a_i^{b_i})\ar[l]
}$$
where the horizontal arrows on the left hand side are quotient maps
while these on the right hand side are inclusions onto open subsets
\end{list}
\end{enumerate}
\end{lemma}
\begin{proof}
Most of the above claims follow by explicit calculations and the
discussion preceding lemma. For example, the set of fixed points of
the action of $S_{r+1}$ is defined in $A^{r+1}$ by equations
$e_0=\cdots=e_r$ and $e_0+\cdots+e_r=0$ hence can be identified with
these points in $A$ whose $(r+1)$-th multiple is zero. The cohomology
of the special fiber of resolution of $A^r/S_{r+1}$, that is of
$F((r+1)^1)$, is known by \cite{KaledinInv}.  The case of
$F((a_i^{b_i}))$ follows because of the uniqueness result from
\cite{FuNamikawa}. This yields that the resolution is locally product
in the sense of \ref{locally-product-resolution}.
\end{proof}

The above lemma provides us with a general layout for computing
cohomology of a generalized Kummer variety. In the present section we
do explicit calculations of the
Poincar\'e polynomial for the generalized Kummer manifolds of dimension $6$
which is a resolution of $A^3/S_4$.

Again, the following lines in typewriter type are in the form {\tt
maxima}, \cite{maxima}.

First, we write the Poincar\'e polynomials of fibers of the resolution over
the respective strata 
{\tt\obeylines F211(t):=1+t\^{}2; F31(t):=1+t\^{}2+t\^{}4; 
F22(t):=(1+t\^{}2)\^{}2; F4(t):=1+t\^{}2+2*t\^{}4+t\^{}6;}

Next we write the polynomials for the surface $A$ and its quotient $A/\Z_2$
where $\Z_2$ acts on $A$ by multiplying by $(-1)$.
{\tt\obeylines 
A(t):=(1+t)\^{}4; B(t):=1+6*t\^{}2+t\^{}4;}

The next line describes cohomology of
$\left(A\times\bP^1\times\bP^1\right)/\Z_2$ where $\Z_2$ acts on $A$
as above while its action on $\bP^1\times\bP^1$ interchanges the
factors. In terms of the $\Z_2$ action the Poincar\'e polynomial of the
product is $(1+\epsilon\cdot t)^4\cdot (1+(1+\epsilon)\cdot t^2 +t^4)$
where $\epsilon^2=1$. Note that we write both, the polynomial of {\tt
A} as well as {\tt F22} depending on the group action. The
invariant part of this action has the following polynomial.

{\tt\obeylines C(t):=1+7*t\^{}2+4*t\^{}3+8*t\^{}4+4*t\^{}5+7*t\^{}6+t\^{}8;}

And finally the Poincar\'e polynomial of $A^3/S_4$ which we calculate by
looking at the invariants of the respective representation, as in
\cite[Part I]{FultonHarris}:
{\tt\obeylines
Q(t):=1+6*t\^{}2+4*t\^{}3+22*t\^{}4+24*t\^{}5+62*t\^{}6+
24*t\^{}7+22*t\^{}8+4*t\^{}9+6*t\^{}10+t\^{}12;}

Now we compute virtual Poincar\'e polynomials of strata of $Y$ and $X$,
denoted by {\tt R} and {\tt S}, respectively. The first are the fixed
points.

{\tt\obeylines
R4(t):=4\^{}4; S4(t):=R4(t)*F4(t);}

Next, take an element of $S_4$ whose decomposition consists of two
cycles of length two, for example $\sigma=(01)(23)$. Its fixed point
set is given by equations $e_0=e_1$, $e_2=e_3$, $e_0+\cdots+e_3=0$,
hence $2\cdot(e_0+e_2)=0$ which makes 16 copies of $A$. The group
$W(2^2))\iso\Z_2$ acts on each of the components by involution, we use
\ref{Poincare-partition-properties}.

{\tt R22(t):=16*B(t)-R4(t); S22(t):=16*C(t)-R4(t)*(1+t\^{}2+t\^{}4);}

The next one is easy, as $W(3,1)$ is trivial.

{\tt R31(t):=A(t)-R4(t); S31(t):=R31(t)*F31(t);}

The fixed point set of $\sigma=(01)$ contains both, the fixed point
set of $(01)(23)$ and of $(012)$, the former one consists of 16 copies
of $A$ and is the fixed point set of the action of $W(2,1,1)\iso\Z_2$,
because $(01)(23)$ is contained in the normalizer of $(01)$. Note that
the action of $W(2,1,1)$ on $F(2,1,1)$ is trivial.

{\tt R211(t):=A(t)*(B(t)-16)-R31(t); S211(t):=R211(t)*F211(t);}

Finally, we write down the general stratum and the Poincar\'e polynomial
of the resolution.

{\tt\obeylines R1111(t):=Q(t)-(R211(t)+R31(t)+R22(t)+R4(t));
S1111(t):=R1111(t); P(t):=S1111(t)+S211(t)+S31(t)+S22(t)+S4(t). }

All together these prove the following:

\begin{proposition}\label{Poincare-Kummer}
The Poincar\'e polynomial of the Beauville's generalized
Kummer variety, which is  given by a crepant resolution of $A^3/S_4$, where
$A$ is a two dimensional torus, is :
$$t^{12}+7\,t^{10}+8\,t^9+51\,t^8+56\,t^7+458\,t^6+56\,t^5+51\,t^4+8
\,t^3+7\,t^2+1$$
\end{proposition}


\section{Building upon a $4$-dimensional abelian manifold.}
\label{general}

In this section we  consider $4$-dimensional abelian varieties, $A$, with the action of 
a finite group $G$. However we will  prescribe the 
action of $G$ only on the complex cohomolgy $H^1(A, \C)$. In other words we fix
the complex representation $\rho _\C: G \ra SL(\a)$ and we do not require that 
it comes from an integral representation $\rho _\Z$.

Moreover we will take the three groups in the theorem \ref{BLS} which have a four dimensional complex representation of type $V\oplus V^*$; by the theorem there exists a local symplectic resolution of $V\oplus V^*/G$
and in the first two cases if $A = S \times S$, where $S$ is an abelian surface, also a global one (see the end of section \ref{sympl}). 

Our main tools in this section are the semismallness property of symplectic resolution, namely \ref{semism}, and the 
Lefschetz fixed point formula, namely \ref{Lefschetz}.

\subsection{The Binary Tetrahedral Group}
\begin{theorem}
Let $G= Q_8 \rtimes \Z_3$  be the binary tetrahedral group acting on a 4-dimensional
complex torus $A$ such that its action on the complex cohomology group $H^1(A, \C)$ is
equivalent to the representations $S_1\oplus S_2$, as in \ref{BLS}. Then the quotient $A/G$ does not admit a
(global) symplectic resolution.
\end{theorem}

\begin{proof} We will use the description of $G$ and its representation given in 
 \cite{Lehn-Sorger}, see also \ref{representations}.  

The following are the non trivial subgroups of $T$: $T >  M_i >H_i$ for $i =1,...,4$, where $H_i$ are the $4$ conjugate $3$-Sylow subgroups and $M_j = \langle-1, H_i \rangle$ are four conjugate subgroups of order $6$;
$ Q_8 >L_j$ where $L_j$ are  $3$ conjugate subgroups of order $4$ generated respectively by $I,J,K$.
Moreover it has a nontrivial center equal to $\langle-1\rangle$ contained in $M_i$ and in $L_j$
 Note that any two of the groups $M_j$  generate $G$ and $\langle -1\rangle$ is contained in any
non-trivial subgroup of $G$ of order different from $3$.

The element $-1$ acts on $H^1(A, \C)$ as the diagonal matrix
with all eigenvalue $-1$ while a generator of $K_i$ acts on $H^1(A, \C)$ 
as the diagonal matrix with eigenvalues $\{-1, -1, e_6,e_6^5\}$, where $e_6$ is a 
$6$-th primitive root of unity. This implies, in particular, that fixed points of $-1$, as well
as of of $K_j$, are isolated;  by the Lefschetz fixed point formula, \ref{Lefschetz}, we
get that there are $2^8$ and, respectively, $2^4$ of them.

For each $j$ the set $Fix(M_j)$ is clearly contained in $Fix(-1)$. We claim
that, if $A/G$ admits a symplectic resolution, then $Fix(-1$)
is the union of $Fix(M_j)$. Indeed,  isolated
quotient symplectic singularities have no crepant resolution in
dimension $> 2$, see (\ref{semism}); so any point of $Fix(-1)$ must be contained
in $Fix(H_i)$ for some $i$,  since the groups $H_i$ are the only subgroups not
containing $-1$ (and therefore they are the only ones whose fixed points sets are not
contained in $Fix(-1)$).

Therefore, if $A/G$ admits a symplectic resolution then we can write
$Fix(-1)$ as a disjoint union of four copies of $(Fix(M_j)\setminus Fix(G))$
and of $Fix(G)$ (note that any two $Fix(M_j)$ intersect along $Fix(G)$). 

Thus, if $s$ is the cardinality of $Fix(G)$, we get the following equality
$$2^8 =|Fix(-1)|=  4(|Fix(M_j)|-s) + s = 4(2^4) - 3s = 2^6 -3s$$
 which is a contradiction.
 \end{proof}

\subsection{The Dihedral group of order $6$.}
 
\begin{theorem} Let $G=D_6=S_3$ be acting on a 4-dimensional complex
torus  $A$ so that its representation on the complex
cohomology is
the sum of two copies of the standard representation. Suppose that the
quotient of the action $A/G$ admits a symplectic resolution of singularities,
$X  \ra A/G$. Then the complex cohomology of X is uniquely determined, that is the
Poincar\'e polynomial is as follows:
$$P_X = 1 + 7t^2 + 8t^3 + 108t^4 + 8t^5 + 7t^6 + t^8.$$
 \end{theorem}

\begin{remark} In section \ref{Got} we have explained how to compute $P_X(t)$ when $X$
is the resolution of $(A^2)^n/S_{n+1}$, $A^2$ being a $2$-dimensional abelian variety; moreover we have done all the computations for  the case $S_4$. One can compute the case of $S_3$ in this way and it will turn out the same formula. So even if one can construct a different resolution of $A/G$, this resolution will have the same cohomological type of the $4$-dimensional case of Beauville's serie.  
\end{remark}

\begin{proof} The structure of conjugacy classes of subgroups of $D_6$ is very
simple, namely there are two proper non-trivial classes of the normal
subgroup $\Z_3$ and of three $\Z_2$ subgroups. By Lefschetz fixed-point
formula, \ref{Lefschetz},  a generator of $\Z_3$ have $81$ isolated fixed points,
which by the lemma \ref{semism}, are fixed points of the whole group. On
the other hand any order two element $t $ acts with fixed point being a
union of, say $m$, abelian surfaces. Since their Weyl group is trivial
the formula in \ref{principle} yields the following result
$$P_X = 1 + (6+m)t^2 + (4+4m)t^3 + (102+6m)t^4 + (4+4m)t^5 + (6+m)t^6 + t^8.$$
We have only to prove that m=1. The argument we use is similar to the one used at the end
of \ref{Kummer}. 

Note that all components of $Fix(t)$ are numerically
equivalent and $Fix(t)^0$ is a subgroup of $A$. This is because the action of $t$ is algebraic and they
are fibers of the map $A \ra A/Fix(t)^0$ (where $Fix(t)^0$ denotes the zero component of
 $Fix(t)$). 
The action of $t$ descends to the two dimensional torus  $A/Fix(t)^0$
as an involution with 16 fixed points of order two.
$Fix(t)$ is a subgroup of $A$ and it descends
to a subgroup of the group of points of order $2$ in $A/Fix(t)^0$, hence $m$
is a power of $2$.

On the other hand, let us take another involution $t'$  in $D_6$  then
$Fix(D_6)=Fix(t)\cap Fix(t')$ and since they intersect tranversaly we get 
$$81=|Fix(D_6)|=m(Fix(t)^0)\cdot m(Fix(t')^0).$$
Thus $m$ is also a power of $3$; all this implies that 
$m=1$.
\end{proof}

\subsection{The Dihedral group of order $8$.}

\begin{theorem} 
Let $G = D_8 =(\Z_2^2) \rtimes \Z_2$ be acting
on a $4$-dimensional complex torus $A$ so that its
representation on the complex cohomology is the sum of two copies of the standard
representation  of $D_8$ as motions of a square. 
Suppose that the quotient of the
action $A/G$ admits a symplectic resolution of singularities,
$X \ra A/G$. Then the complex cohomology of X is uniquely determined, that is
the Poincar\'e polynomial is as follows:
$$P_X(t) = t^8  + 23t^6  + 276t^4  + 23t^2  + 1.$$
\end{theorem}

\begin{remark} Similarly to the previous case note that if $A= S \times S$, with $S$ an  abelian surface, the symplectic resolution exists, see \ref{sympl}. The $P_X(t)$ we obtain is therefore the Poincar\'e polynomial of this manifold, other possible resolutions will have the same polynomial.
\end{remark}

\begin{proof} 
Let us recall some trivial fact regarding the standard
representation of $D_8$ and its conjugacy classes. The center of the
group consists of $\langle -1 \rangle$. There are two conjugate elements of order $4$, namely 
$\pm \left(\begin{array}{rr}0&1\\-1&0\end{array}\right)$, and
two classes of elements of order $2$, namely 
$\pm A:= \pm \left(\begin{array}{rr}1&0\\0&-1\end{array}\right)$ and 
$\pm B:=\pm \left(\begin{array}{rr}0&1\\1&0\end{array}\right)$. 

From Lefschetz fixed point formula, \ref{Lefschetz}, 
we compute that the order $4$ elements have $2^4=16$ fixed points, which,
by the above lemma \ref{semism} are the fixed points of the whole group;
on the other hand $-I$   has
$2^8=256$ fixed points, including the previous $2^4$. 
The other  elements  of order two have fixed points set being union of abelian
surfaces; let us assume that $Fix(\pm A)$ has $a$
components and $Fix(\pm B)$ has $b$ components.

Again, by the lemma \ref{semism} the points $Fix(-I)\setminus Fix(G)$ are divided
among $Fix(\pm A)$ and $Fix(\pm B)$. On the other hand
the normalizer of each of the order $2$ element (different from $-I$)  is a group
generated by the element itself and $-I$. Therefore the Weyl group acts on
each $2$-dimensional component of such an element by involution, with $16$
points fixed. Thus we can count the points in the quotient whose
isotropy group is non-cyclic and different from the whole group $G$ 
to get the following identity

$|Fix(-I)|-|Fix(G)| = (a+b)\cdot 16 - 2\cdot|Fix G|$, hence $a+b = 17$.

On the other hand, as in the proof of the previuos
theorem, the numbers $a$ and $b$  are powers of $2$, therefore $a = 1$ and $b = 16$ (up
to choice of the conjugacy classes). 

Note that, as in the proof of the previous theorem, we could compute the intersection of the connected
components of each of the fixed point sets for both conjugacy
classes. For instance we get $ab \cdot Fix(A)^0\cdot Fix(B)^0 = |Fix(G)| = 16,$
hence the intersection is equal to $1$.

Let us then compute the cohomology:
the Poincare polynomial of the $4$-dimensional stratum, $X([1])$, is
$$(1 + 6t^2 + 22t^4 + 6t^6 + t^8) - 17(1 + 6t^2 + t^4 - 16) - 136;$$
the polynomial of the $3$ -dimensional strata is
$$17(1 + 6t^2 + t^4 - 16)(1 + t^2)$$
and the one of the $2$-dimensional strata is
$$120(1 + 2t^2 + t^4) + 16(1 + 2t^2 + 2t^4).$$
The Poincar\'e polynomial of $X$ is therefore
$$(1+6t^2+22t^4+6t^6+t^8)-17(1+6t^2+t^4-16)-136+
17(1+6t^2+t^4-16)(1+t^2)+$$
$$120(1+2t^2+t^4)+16(1+2t^2+2t^4)$$
which, after the simplification, becomes
$$t^8  + 23t^6  + 276t^4  + 23t^2  + 1$$
\end{proof}
 

\begin{thebibliography}{BPVdV84}

\bibitem[Bea83]{BeauvilleJDG}
Arnaud Beauville.
\newblock Vari\'et\'es {K}\"ahleriennes dont la premi\`ere classe de {C}hern
  est nulle.
\newblock {\em J. Differential Geom.}, 18(4):755--782 (1984), 1983.

\bibitem[Bel07]{Bell}
Gwyn Bellamy.
\newblock On singular calogero-moser spaces, http://arxiv.org/abs/0707.3694,
  2007.

\bibitem[BL04]{Birkenhake-Lange}
Christina Birkenhake and Herbert Lange.
\newblock {\em Complex abelian varieties}, volume 302 of {\em Grundlehren der
  Mathematischen Wissenschaften [Fundamental Principles of Mathematical
  Sciences]}.
\newblock Springer-Verlag, Berlin, second edition, 2004.

\bibitem[BM94]{BertinMarkushevich}
J.~Bertin and D.~Markushevich.
\newblock Singularit\'es quotients non ab\'eliennes de dimension {$3$} et
  vari\'et\'es de {C}alabi-{Y}au.
\newblock {\em Math. Ann.}, 299(1):105--116, 1994.

\bibitem[Bou68]{BourbakiLie}
N.~Bourbaki.
\newblock {\em \'{E}l\'ements de math\'ematique. {F}asc. {XXXIV}. {G}roupes et
  alg\`ebres de {L}ie. {C}hapitre {IV}: {G}roupes de {C}oxeter et syst\`emes de
  {T}its. {C}hapitre {V}: {G}roupes engendr\'es par des r\'eflexions.
  {C}hapitre {VI}: syst\`emes de racines}.
\newblock Actualit\'es Scientifiques et Industrielles, No. 1337. Hermann,
  Paris, 1968.

\bibitem[BPVdV84]{BarthPetersVandeVen}
W.~Barth, C.~Peters, and A.~Van~de Ven.
\newblock {\em Compact complex surfaces}, volume~4 of {\em Ergebnisse der
  Mathematik und ihrer Grenzgebiete (3) [Results in Mathematics and Related
  Areas (3)]}.
\newblock Springer-Verlag, Berlin, 1984.

\bibitem[Bre72]{Bredon}
Glen~E. Bredon.
\newblock {\em Introduction to compact transformation groups}.
\newblock Academic Press, New York, 1972.
\newblock Pure and Applied Mathematics, Vol. 46.

\bibitem[CH07]{CynkHulek}
S.~Cynk and K.~Hulek.
\newblock Higher-dimensional modular {C}alabi-{Y}au manifolds.
\newblock {\em Canad. Math. Bull.}, 50(4):486--503, 2007.

\bibitem[CR62]{CurtisReiner}
Charles~W. Curtis and Irving Reiner.
\newblock {\em Representation theory of finite groups and associative
  algebras}.
\newblock Pure and Applied Mathematics, Vol. XI. Interscience Publishers, a
  division of John Wiley \& Sons, New York-London, 1962.

\bibitem[CS07]{CynkSchuett}
Slawomir Cynk and Matthias Schuett.
\newblock Generalised kummer constructions and {W}eil restrictions,
  {http://arxiv.org/abs/0710.4565}, 2007.

\bibitem[Deb99]{Debarre}
Olivier Debarre.
\newblock On the {E}uler characteristic of generalized {K}ummer varieties.
\newblock {\em Amer. J. Math.}, 121(3):577--586, 1999.

\bibitem[DHZ06]{DaisHenkZiegler}
Dimitrios~I. Dais, Martin Henk, and G{\"u}nter~M. Ziegler.
\newblock On the existence of crepant resolutions of {G}orenstein abelian
  quotient singularities in dimensions {$\ge4$}.
\newblock In {\em Algebraic and geometric combinatorics}, volume 423 of {\em
  Contemp. Math.}, pages 125--193. Amer. Math. Soc., Providence, RI, 2006.

\bibitem[Don08]{Donten}
Maria Donten.
\newblock On kummer 3-folds, http://arxiv.org/abs/0812.3758, 2008.

\bibitem[Dur79]{Durfee15}
Alan~H. Durfee.
\newblock Fifteen characterizations of rational double points and simple
  critical points.
\newblock {\em Enseign. Math. (2)}, 25(1-2):131--163, 1979.

\bibitem[FH91]{FultonHarris}
William Fulton and Joe Harris.
\newblock {\em Representation theory}, volume 129 of {\em Graduate Texts in
  Mathematics}.
\newblock Springer-Verlag, New York, 1991.
\newblock A first course, Readings in Mathematics.

\bibitem[FN04]{FuNamikawa}
Baohua Fu and Yoshinori Namikawa.
\newblock Uniqueness of crepant resolutions and symplectic singularities.
\newblock {\em Ann. Inst. Fourier (Grenoble)}, 54(1):1--19, 2004.

\bibitem[Fu06]{FuSurvey}
Baohua Fu.
\newblock A survey on symplectic singularities and symplectic resolutions.
\newblock {\em Ann. Math. Blaise Pascal}, 13(2):209--236, 2006.

\bibitem[Fuj83]{Fujiki}
Akira Fujiki.
\newblock On primitively symplectic compact {K}\"ahler {$V$}-manifolds of
  dimension four.
\newblock In {\em Classification of algebraic and analytic manifolds ({K}atata,
  1982)}, volume~39 of {\em Progr. Math.}, pages 71--250. Birkh\"auser Boston,
  Boston, MA, 1983.

\bibitem[Ful93]{FultonToric}
William Fulton.
\newblock {\em Introduction to toric varieties}, volume 131 of {\em Annals of
  Mathematics Studies}.
\newblock Princeton University Press, Princeton, NJ, 1993.
\newblock ,The William H. Roever Lectures in Geometry.

\bibitem[GK04]{GiKa}
Victor Ginzburg and Dmitry Kaledin.
\newblock Poisson deformations of symplectic quotient singularities.
\newblock {\em Adv. Math.}, 186(1):1--57, 2004.

\bibitem[G{\"o}t93]{Gottsche}
Lothar G{\"o}ttsche.
\newblock {\em Hilbertschemata nulldimensionaler {U}nterschemata glatter
  {V}ariet\"aten}.
\newblock Bonner Mathematische Schriften [Bonn Mathematical Publications], 243.
  Universit\"at Bonn Mathematisches Institut, Bonn, 1993.
\newblock Dissertation, Rheinische Friedrich-Wilhelms-Universit\"at Bonn, Bonn,
  1991.

\bibitem[G{\"o}t94]{Gottschebook}
Lothar G{\"o}ttsche.
\newblock {\em Hilbert schemes of zero-dimensional subschemes of smooth
  varieties}, volume 1572 of {\em Lecture Notes in Mathematics}.
\newblock Springer-Verlag, Berlin, 1994.

\bibitem[GS93]{GottscheSoergel}
Lothar G{\"o}ttsche and Wolfgang Soergel.
\newblock Perverse sheaves and the cohomology of {H}ilbert schemes of smooth
  algebraic surfaces.
\newblock {\em Math. Ann.}, 296(2):235--245, 1993.

\bibitem[HH90]{Hirzebruch}
Friedrich Hirzebruch and Thomas H{\"o}fer.
\newblock On the {E}uler number of an orbifold.
\newblock {\em Math. Ann.}, 286(1-3):255--260, 1990.

\bibitem[JK81]{JamesKerber}
Gordon James and Adalbert Kerber.
\newblock {\em The representation theory of the symmetric group}, volume~16 of
  {\em Encyclopedia of Mathematics and its Applications}.
\newblock Addison-Wesley Publishing Co., Reading, Mass., 1981.
\newblock With a foreword by P. M. Cohn, With an introduction by Gilbert de B.
  Robinson.

\bibitem[Kal02]{KaledinInv}
D.~Kaledin.
\newblock Mc{K}ay correspondence for symplectic quotient singularities.
\newblock {\em Invent. Math.}, 148(1):151--175, 2002.

\bibitem[Kal03]{KaledinSelecta}
D.~Kaledin.
\newblock On crepant resolutions of symplectic quotient singularities.
\newblock {\em Selecta Math. (N.S.)}, 9(4):529--555, 2003.

\bibitem[KP02]{FinGpsSurveyAMM}
James Kuzmanovich and Andrey Pavlichenkov.
\newblock Finite groups of matrices whose entries are integers.
\newblock {\em Amer. Math. Monthly}, 109(2):173--186, 2002.

\bibitem[Kum75]{Kummer}
Ernst~Eduard Kummer.
\newblock {\em {\"Uber die Fl\"achen vierten Grades mit sechzehn singul\"aren
  Punkten}}.
\newblock Springer-Verlag, Berlin, 1975.
\newblock Collected papers, Volume II: Function theory, geometry and
  miscellaneous, Edited and with a foreward by Andr\'e Weil.

\bibitem[LS08]{Lehn-Sorger}
Manfred Lehn and Christoph Sorger.
\newblock A symplectic resolution for the binary tetrahedral group,
  http://arxiv.org/abs/0810.3225, 2008.

\bibitem[New72]{Newman}
Morris Newman.
\newblock {\em Integral matrices}.
\newblock Academic Press, New York, 1972.
\newblock Pure and Applied Mathematics, Vol. 45.

\bibitem[NW02]{Nieper2}
Marc~A. Nieper-Wi{\ss}kirchen.
\newblock On the {C}hern numbers of generalised {K}ummer varieties.
\newblock {\em Math. Res. Lett.}, 9(5-6):597--606, 2002.

\bibitem[NW04]{Nieper1}
Marc~A. Nieper-Wi{\ss}kirchen.
\newblock On the elliptic genus of generalised {K}ummer varieties.
\newblock {\em Math. Ann.}, 330(2):201--213, 2004.

\bibitem[OS01]{Oguiso}
Keiji Oguiso and Jun Sakurai.
\newblock Calabi-{Y}au threefolds of quotient type.
\newblock {\em Asian J. Math.}, 5(1):43--77, 2001.

\bibitem[PR05]{Paranjape}
Kapil Paranjape and Dinakar Ramakrishnan.
\newblock Quotients of {$E\sp n$} by {$\mathfrak a\sb {n+1}$} and
  {C}alabi-{Y}au manifolds.
\newblock In {\em Algebra and number theory}, pages 90--98. Hindustan Book
  Agency, Delhi, 2005.

\bibitem[Rei87]{ReidYPG}
Miles Reid.
\newblock Young person's guide to canonical singularities.
\newblock In {\em Algebraic geometry, Bowdoin, 1985 (Brunswick, Maine, 1985)},
  volume~46 of {\em Proc. Sympos. Pure Math.}, pages 345--414. Amer. Math.
  Soc., Providence, RI, 1987.

\bibitem[Rei02]{ReidBourbaki}
Miles Reid.
\newblock La correspondance de {M}c{K}ay.
\newblock {\em Ast\'erisque}, (276):53--72, 2002.
\newblock S\'eminaire Bourbaki, Vol.\ 1999/2000.

\bibitem[Saw04]{SawonPhD}
Justin Sawon.
\newblock Rozansky-{W}itten invariants of hyperk\"ahler manifolds, {PhD}
  {C}ambridge, 2004.

\bibitem[{Sch}07]{maxima}
{William} {Schelter et al}.
\newblock Maxima, a computer algebra system, available at
  http://maxima.sourceforge.net/, 1982-2007.

\bibitem[Tot02]{Totaro}
B.~Totaro.
\newblock Topology of singular algebraic varieties.
\newblock In {\em Proceedings of the International Congress of Mathematicians,
  Vol. II (Beijing, 2002)}, pages 533--541, Beijing, 2002. Higher Ed. Press.

\bibitem[Ver00]{VerbitskyAsian}
Misha Verbitsky.
\newblock Holomorphic symplectic geometry and orbifold singularities.
\newblock {\em Asian J. Math.}, 4(3):553--563, 2000.

\end{thebibliography}

\end{document}